\documentclass[11pt,leqno]{article}%
\usepackage{amssymb,bbm,amsmath,amsfonts,amsthm}
\usepackage{mathtools}
\usepackage{vmargin}

\usepackage[dvipsnames]{xcolor}
 
\usepackage{mathrsfs}
\usepackage{macros}

\usepackage{enumerate}

\usepackage[linktocpage,colorlinks=true,raiselinks]{hyperref}
\hypersetup{urlcolor=blue, citecolor=red, linkcolor=blue}

\usepackage{authblk}

\makeatletter

\@addtoreset{equation}{section}

\begin{document}
\title{\bf Turbulent flows  as generalized Kelvin-Voigt materials: modeling and analysis}
\author[1]{Cherif Amrouche} 
\author[2]{Luigi C. Berselli} 
\author[3]{Roger Lewandowski}
\author[4]{Dinh Duong Nguyen}
\affil[1]{Laboratoire de Math\'ematiques et leurs Applications, UMR CNRS 5142, Universit\'e de Pau et des Pays de
  l'Adour, France, E-mail:
  cherif.amrouche@univ-pau.fr} 
\affil[2]{ Universit\`a di Pisa, Dipartimento di Matematica, Via Buonarroti~1/c, I-56127
  Pisa, Italy, E-mail: luigi.carlo.berselli@unipi.it} 
\affil[3,4]{IRMAR, UMR CNRS 6625, University of Rennes 1 and FLUMINANCE Team, INRIA
  Rennes, France
  \\
  E-mail: Roger.Lewandowski@univ-rennes1.fr, dinh-duong.nguyen@univ-rennes1.fr}
\date{}
\maketitle
\begin{abstract} 
  We model a 3D turbulent fluid, evolving toward a statistical equilibrium, by adding to
  the equations for the mean field $(\vv, p)$ a term like $-\alpha \div (\ell(\x) D
  \vv_t)$. This is of the Kelvin-Voigt form, where the Prandtl mixing length $\ell$ is not
  constant and vanishes at the solid walls. We get estimates for velocity $\vv$ in
  $L^\infty_t H^1_x \cap W^{1,2}_t H^{1/2}_x$, that allow us to prove the existence and
  uniqueness of a regular-weak solutions $(\vv, p)$ to the resulting system, for a given
  fixed eddy viscosity. We then prove a structural compactness result that highlights the
  robustness of the model. This allows us to pass to the limit in the quadratic source
  term in the equation for the turbulent kinetic energy $k$, which yields the existence of
  a weak solution to the corresponding Reynolds Averaged Navier-Stokes system satisfied by
  $(\vv, p, k)$.
\end{abstract} 
\smallskip

\section{Introduction}
The purpose of this paper is to model incompressible turbulent flows as generalized
viscoelastic materials involving the Prandtl mixing length $\ell$ (see
in~\cite{MR2605601}), to show the existence and uniqueness of regular-weak solutions to
the resulting  system  of Partial Differential Equations (PDE),
\begin{equation}
\label{eq:Voigt}
\left\{  
    \begin{aligned}
        \vv_{t} - \alpha \div (\ell (\x) D \vv_t) + \div (\vv \otimes \vv) - \nu\Delta \vv -
        \div(\nut D \vv) + \nabla 
        p &= \fv, \qquad       
        \\
        \nabla\cdot \vv &= 0, \qquad 
    \end{aligned}
\right.
\end{equation} 
for a given turbulent viscosity (eddy viscosity) $\nut$. We then study the existence weak solutions to the
corresponding NSTKE\footnote{RANS = Reynolds Averaged Navier-Stokes. NSTKE =
  Navier-Stokes-Turbulent-Kinetic-Energy. NSTKE model is a specific RANS model.} system,
\begin{equation} \label{eq:Voigt_NSTKE}
\left\{  
\begin{array}{l}
  \vv_{t} - \alpha \div (\ell (\x) D \vv_t) + \div (\vv \otimes \vv) - \nu\Delta \vv-
  \div(\nut(k) D \vv) + \nabla 
  p = \fv, \qquad
  \\      
  \nabla\cdot \vv = 0, 
  \\
  k_t + \vv \cdot \g k - \div( \mut (k) \g k)  =  \nut(k) | D \vv |^2  -  (\ell + \eta)^{-1} k \sqrt{|k|},
  \end{array}
\right.
\end{equation} 
where, to fix the notation,
\begin{itemize} 
\item $\vv$ is the mean velocity\footnote{Usually, the mean velocity is denoted by
    $\vm$. Throughout the paper we omit the over-line for simplicity, except in
    Section~\ref{sec:modelling}, devoted to turbulence modelling.}, $\zoom \vv_t =
  {\partial \vv \over \partial t}$;
\item $D \vv = {1 \over 2} (\g \vv + \g \vv^t)$ is the deformation stress;
\item $p$ is the mean pressure;
\item $k$ is the Turbulent Kinetic Energy (TKE);
\item $\nu >0$ is the  kinematic viscosity, $\nut$ the eddy viscosity;
\item $\mu_t$ is the eddy diffusion and $\eta >0$ is a small constant;
\item the length scale $\alpha$ is that of the boundary layer, given by the relation 
\begin{equation} 
\label{eq:alpha} \alpha =  {\nu \over u_\star},
\end{equation}
here $u_\star$ is the so called friction velocity (see~\cite{CL14});
\item  $\fv$ is a given source term. 
\end{itemize}
As usual, the systems are set in a bounded Lipschitz domain $\Om \subset \R^3$. The mixing
length $\ell = \ell (\x) \ge 0$ is defined over $\Om$ and, according to well known
physical laws (see~\eqref{eq:Prandtl_Obukhov} and~\eqref{eq:van_driest} below), $\ell \in
C^1(\overline \Om)$ and vanishes at the boundary $\Ga = \p \Om$ as follows:
\begin{equation} 
\label{eq:law_for_ell_1} 
\ell (\x) \simeq d(\x, \Ga) = \rho(\x), \quad \hbox{ when } \quad \x \to \Ga, \, \x \in
\Om,
\end{equation} 
where $d(\x, \Ga)$ denotes the distance of the point $\x$ from the boundary.

Model~\eqref{eq:Voigt} is close to viscoelastic materials models, given by the
Kelvin-Voigt relation: 
\begin{equation} 
  \label{eq:law_KV} 
  {\boldsymbol \sigma} = E\, {\boldsymbol \E} + \eta\, {\boldsymbol\E}_t, 
\end{equation}
where ${\boldsymbol \sigma}$ denotes the Cauchy stress tensor and ${\boldsymbol \E}$ the
strain-rate tensor. In this case, $E$ is the modulus of elasticity and $\eta$ the
viscosity (see for instance Germain~\cite{MR0154462} or Gurtin~\cite{MR636255}). In fluid
mechanics, ${\boldsymbol \E} = D\vv$, and this model is used to describe some non
Newtonian fluids, such as lubricants. For such flows, the law~\eqref{eq:law_KV} becomes
\begin{equation*} 
{\boldsymbol \sigma} = - p\, {\rm Id} + \mu D \vv +
\gamma^2 D \vv_t, 
\end{equation*} that yields the incompressible Navier-Stokes-Voigt equations:
\begin{equation}
  \label{eq:NSVE} 
  \left\{
    \begin{aligned}
      \vv_{t} - \gamma^2 \Delta \vv_t + \div (\vv \otimes \vv) -\nu\Delta \vv+\nabla
      p &= \fv,       
      \\ 
    \nabla\cdot \vv &= 0.
  \end{aligned}
\right.
\end{equation} 
Mathematical investigations about system~\eqref{eq:NSVE} were first carried out by
A. Oskolkov, who proved the existence and uniqueness of weak and strong solutions in some
particular sense, see~\cite{MR0377311, MR579485}. Then, several mathematical problems
raised by~\eqref{eq:NSVE} have been studied by Titi {\sl et al.}~\cite{MR2570790,
  MR2655910, MR2629486}, making a clear relation between Navier-Stokes-Voigt and
turbulence modeling. In addition, in~\cite{MR2660874} Larios \& Titi showed the
connection between the Navier-Stokes-Voigt equations and the simplified Bardina's model
introduced by Layton \& Lewandowski~\cite{MR2172198}, designed as a Large-Eddy simulation
model. In Berselli, Kim, and Rebholz~\cite{BKR2016} an interpretation of the
Navier-Stokes-Voigt equations in terms of approximate deconvolution models is also given.
\medskip

In this paper we connect the Prandtl-Smagorinsky's model to the Turbulent Kinetic Energy
(TKE) model to calculate the eddy viscosity $\nut$. To make it clear, let $\reyn$ denotes
the Reynolds stress. We will show how, combining the energy inequality with the equation
satisfied by $k$ (without any closure assumption), we are led to set --in certain specific
regimes, such as the convergence to stable statistical states
see~\eqref{eq:transfer_comp}-- the following constitutive law
\begin{equation*} 
  \reyn = - \alpha \ell\, D  \vv_t - \nut D \vv + {2 \over 3} k\, {\rm Id}, 
\end{equation*} 
instead of the usual one
\begin{equation*}
    \reyn = - \nut D \vv + {2 \over 3} k\, {\rm Id}.
  \end{equation*}
  This yields the PDE system~\eqref{eq:Voigt} including the term $-\alpha \div (\ell D
  \vv_t)$, and then also the NSTKE system~\eqref{eq:Voigt_NSTKE} after having performed
  the usual closure procedure about $k$, where $\nut = \nut(k) = \ell \sqrt k$.
  \medskip

  Turning to the analysis of the systems, we observe that according to
  assumption~\eqref{eq:law_for_ell_1} about the mixing length $\ell$, the additional
  generalized Kelvin-Voigt term $- \alpha \div (\ell D \vv_t)$ enforces for the equations
  a natural functional structure in the space $H^{1/2} (\Om)= [H^1(\Om), L^2(\Om)]_{1/2}$,
  cf.~Lions \& Magenes~\cite{MR0350177}, which is a critical scaling-invariant space for
  the Navier-Stokes equations.  In particular, we obtain for the velocity sharp estimates
  in $W^{1,2} (0,T; H^{1/2}(\Om)^3)$, as well as in $L^\infty(0,T; H_0^1(\Om)^3)$. We are
  then able to prove the existence and uniqueness of regular-weak solution
  to~\eqref{eq:Voigt} (see Theorem~\ref{thm:existence} and
  Remark~\ref{rem:generalisation}).  \medskip

  However, we believe that the most interesting result of this paper is the compactness
  result we prove in Lemma~\ref{thm:compact}. We consider an eddy viscosities sequence
  $(\nut^n)_{n \in \N}$ which is bounded in $L^\infty ([0, \infty[ \times \Om)$ and in
  addition converges {\sl a.e.}  to $\nut$ in $[0, \infty[ \times \Om$ as $n \to
  \infty$. We also show that the corresponding regular-weak sequence of solution
  $(\vv^n)_{n \in \N}$ converges, in some sense, to the regular-weak solution $\vv$ of the
  limit problem with $\nut$ as eddy viscosity. Moreover, we get the convergence of the
  energies, that is $\nut^n | D \vv^n|^2 \to \nut | D \vv |^2$ in the sense of the
  measures.  \medskip

  This compactness result allows us to prove the existence of a solution to the
  NSTKE-Voigt system~\eqref{eq:Voigt_NSTKE} (see Theorem~\ref{thm:existence_NSTKE}
  below). We stress that the usual system coupling $\vv$, $p$. and $k$ only yields a
  variational inequality for $k$ when passing to the limit in the equations, because of
  the lack of strong convergence of the energies (see~\cite{CL14, MR1418142}). This
  observation makes Theorem~\ref{thm:existence_NSTKE} a very interesting and original
  result.  \bigskip

  \textbf{Plan of the paper.} The paper is organized as follows:
  Section~\ref{sec:modelling} is devoted to modeling and to explain the motivations for
  the systems of PDE we study. Then, in Section~\ref{sec:functional-setting} we use
  functional analysis and interpolation theory to provide estimates in various spaces,
  especially in $H^{1/2}(\Omega)$. The proof of the existence and uniqueness
  results for the generalized Navier-Stokes-Voigt equations~\eqref{eq:Voigt}
  and~\eqref{eq:Gen-VoigtNSE} is developed in
  Section~\ref{sec:generalized-Voigt}. Finally, the compactness result and analysis of the
  NSTKE-Voigt system is performed in Section~\ref{sec:TKE}.

  \section{Kelvin-Voigt modeling for turbulent flows} \label{sec:modelling} 
  In this section (and only in this section) $\vv$ and $p$ denote the velocity and
  pressure of the fluid respectively (and not the mean fields unlike in the rest of the
  paper). Hence, the couple $(\vv,p)$ solves the Navier-Stokes equations,
\begin{equation} \label{eq:NSE}
\left\{  
    \begin{aligned}
        \vv_{t} + \div (\vv \otimes \vv) -\nu\Delta \vv+\nabla
        p &= \fv, \qquad       
        \\
        \nabla \cdot \vv &= 0. \qquad 
    \end{aligned}
\right.
\end{equation} 
We first recall some results about basic turbulence modeling to derive the equation for
the mean $\vm$ and to define the Reynolds stress $\reyn$. Then, we show how --when
simultaneously using the Prandtl-Smagorinsky and the turbulent kinetic energy models and
the equation satisfied by the TKE-- we get the additional term $-\alpha \div (\ell
D\vm_t)$ in the equation for $\vm$ in specific regimes, such as the convergence to a
statistical equilibrium (see Remark~\ref{rem:stable_conv}).  \medskip

We wish to mention that a very close modeling process has been previously performed in
Yao, Layton, and Zhao~\cite{RLZ18}. The latter paper gave to us some inspiration for this
modelling procedure we develop here. One main difference is that we study the TKE
equations, while in their paper, Yao, Layton, and Zhao considered a rotational structure,
without involving the equation for the turbulent kinetic energy. Moreover, they were
looking at back-scatter terms, so that our point of view and interpretation are --at the
very end-- rather different.
\subsection{Recalls of basic turbulence modeling} 
According to the Reynolds decomposition, $\vv$ and $p$ are decomposed as the sum of their
mean and fluctuation
 \begin{equation*}
\vv = \vm + \vv', \quad \text{and}\quad p = \plm + p',
\end{equation*}
where the mean filter is linear, commutes with any differential operator (namely
$D\overline \psi = \overline {D \psi}$), and it is idempotent (that is
$\overline{\overline \psi} = \overline \psi$). From this, one gets the relation
 \begin{equation*}
\overline{ \vv \otimes \vv } = \vm
 \otimes \vm + \reyn,
\end{equation*}
where the Reynolds stress $\reyn$ is given by 
\begin{equation*}
\reyn = \overline {\vv' \otimes \vv' }.
\end{equation*}
Therefore, applying the mean operator to the NSE~\eqref{eq:NSE} yields
\begin{equation} \label{eq:NSEM}
\left\{  
    \begin{aligned}
        \vm_{t}+ \div (\vm\otimes \vm) - \nu\Delta \vm + \div \reyn +\nabla
        \plm &= \overline{\fv}, \qquad       
        \\
        \nabla \cdot \vm &= 0. \qquad 
    \end{aligned}
\right.
\end{equation} 
To ``close''~\eqref{eq:NSEM}, one must express $\reyn$ in terms of mean quantities. As we
already said in the introduction, the Boussinesq assumption~\cite{Bou1877} yields
\begin{equation*}
\reyn = - \nut D \vm + {2 \over 3} k\, {\rm Id},
\end{equation*}
where we recall that $\nut$ is the eddy viscosity, $k = {1 \over 2} \overline{ | \vv'
  |^2}$ the turbulent kinetic energy (TKE), and $D\vv = {1 \over 2}(\g \vv + \g \vv^T)$
the deformation tensor.  \medskip

The challenge in turbulence modelling is the determination of $\nut$. In what follows, we
combine the Prandtl-Smagorinsky's model with
\begin{equation}
\label{eq:smago} \nut = \ell \sqrt{\alpha \ell} | D\vm |,
\end{equation}
where the boundary layer length scale $\alpha$ is given by~\eqref{eq:alpha},
and the NSTKE model with $\nut$ is given by
\begin{equation}
\label{eq:eddy_k}  
\nut = \ell \sqrt {k}.
\end{equation}

Dimensionless constants may be involved in the above equations. We have set them equal to
1 for the sake of simplicity.  Both models involve the Prandtl mixing length $\ell$. In
the case of a flow over a plate $\Om = \R^2 \times \{ z > 0 \}$, one finds in
Obukhov~\cite{MR0023169} the following law
\begin{equation} \label{eq:Prandtl_Obukhov} \ell = \ell (z) = \kappa z, \end{equation}
where $\kappa \in [0.35, 0.42]$ is  the Von K\'arm\'an constant. A more sophisticated
formula, however very popular especially for the use in the computation of the turbulent
channel flow, can be found in Van Driest~\cite{VD56}:
\begin{equation} \label{eq:van_driest} \ell = \ell (z) = \kappa z (1- e^{-z/A}), \end{equation} 
where $A$ depends on the oscillations of the plate and $\nu$. Alternative formulas are
provided in~\cite{LPMC18}. In all cases, $\ell$ satisfies the
law~\eqref{eq:law_for_ell_1}.
\subsection{Modelling process} 
We start from the natural energy inequality deduced from the equation~\eqref{eq:NSEM} by
usual integration by parts, at
any time positive time $t$\footnote{We use $\|\cdot\|$ for the $L^2$-norm in this section.}, 
\begin{equation}
  \label{eq:ine_energy} 
  {1 \over 2} {d \over dt} \| \vm (t) \|^2 +
  \nu \| \g \vm (t) \|^2 + \langle \div \reyn, \vm (t) \rangle \le \langle \overline \fv
  (t), \vm (t) \rangle.
\end{equation}
We aim to evaluate the contribution of the term 
\begin{equation*}
\mathscr{T}(t) = \langle \div \reyn,
\vm (t) \rangle,
\end{equation*}
which will be deduced from the equation satisfied by $k$ (see~\cite[Sec.~4.4.1]{CL14})
\begin{equation*}
  \p_t k
  + \vm \cdot \g k + \div \overline{e' \vv'} = - \reyn : \g \vm - \E + \overline{ \fv' \cdot
    \vv'},
\end{equation*}
where $e = k + e' = {1 \over 2} |\vv'|^2$ denotes the kinetic energy of the fluctuations,
and $\E$ is the turbulent dissipation,
\begin{equation*}
\E := \nu \overline{ | D \vv' |^2}.
\end{equation*}
Integrating formally this equation in space, leaving apart eventual boundary condition
issues, leads to 
\begin{equation*}
{d \over dt} \int_\Om k(t) = \mathscr{T}(t) - \int_\Om \E(t) +
\overline{\langle \fv' , \vv' \rangle},
\end{equation*}
that we insert in the inequality~\eqref{eq:ine_energy} to obtain 
\begin{equation}
\label{eq:energy_2} \langle \vm_t,
\vm \rangle + {d \over dt} \int_\Om k+ \nu \| \g \vm (t) \|^2 + \int_\Om \E (t) \le
\langle \overline \fv (t), \vm (t) \rangle + \overline{\langle \fv'(t) , \vv' (t)\rangle}.
\end{equation}
In order to eliminate the term $\zoom {d \over dt} \int_\Om k$ from~\eqref{eq:energy_2},
we enforce equality between the Prandtl-Smagorinsky's model~\eqref{eq:smago} and the NSTKE
one~\eqref{eq:eddy_k}, which leads to the closure equality
\begin{equation}
\label{eq:closure_assumption1} k = \alpha\ell | D \vm |^2.
\end{equation}
Then, by using~\eqref{eq:closure_assumption1}, we get the formal identity
\begin{equation}
  \label{eq:closure_assumption} 
  {d \over dt} \int_\Om k = \alpha {d \over dt} \int_\Om
  \ell | D \vm |^2 = -\langle \alpha \div ( \ell D \vm_t), \vm \rangle.
\end{equation}
Finally, we combine~\eqref{eq:energy_2} with~\eqref{eq:closure_assumption}, which leads to
the inequality
\begin{equation}
\label{eq:energy_3} 
\langle \vm_t - \alpha \div (\ell D \vm_t) - \nu
\Delta \vm, \vm \rangle + \int_\Om \E (t) \le \langle \overline \fv (t), \vm (t) \rangle +
\overline{\langle \fv'(t) , \vv' (t)\rangle}.
\end{equation}
Combining~\eqref{eq:ine_energy} to~\eqref {eq:energy_3}, suggests to put
\begin{equation*}
  \reyn = - \alpha \ell D \vm_t - \nut D \vm + {2 \over 3} k {\rm    Id},
\end{equation*}
and yields the following energy inequality 
\begin{equation} 
  \label{eq:ine_energy_5} 
  {1 \over 2} {d \over dt} ( \| \vm (t) \|^2
  +\alpha \| \sqrt \ell D \vm \|^2) + \nu \| \g \vm (t) \|^2 + \| \sqrt { \nut } D \vm
  \|^2 \le \langle \overline \fv (t), \vm (t) \rangle.  
\end{equation} 
Comparing inequalities~\eqref{eq:ine_energy_5} and~\eqref{eq:energy_3}, we see that all
this makes sense when: 
\begin{equation}
  \label{eq:transfer_comp}
  \| \sqrt { \nut } D \vm \|^2 + \overline{\langle
    \fv'(t) , \vv' (t)\rangle} \le \| \sqrt {\E (t)} \|^2,
 \end{equation} 
 and in this case the system satisfied by $\vm$ becomes
\begin{equation*} 
  \left\{  
    \begin{aligned}
      \vm_{t}- \alpha \div (\ell D \vm_t) + \div (\vm\otimes \vm) - \nu\Delta \vm  - \div
      (\nut D \vm ) +\nabla 
      \left   (\plm +{2 \over 3}  k \right ) &= \overline \fv, \qquad       
      \\
      \nabla\cdot \vm & = 0. \qquad 
    \end{aligned}
  \right.
\end{equation*} 

\begin{Remark} 
  When $\ell$ is constant and equal to $2 \alpha$ (to set ideas), and as
  $\div \vm_t = 0$, we have $\alpha \div (\ell D \vm_t) = \alpha^2 \Delta
  \vm_t$. Therefore, we get by this way the Kelvin-Voigt term involved in
  Equation~\eqref{eq:NSVE}.
\end{Remark}
\begin{Remark}
  \label{rem:stable_conv}  
  Condition~\eqref{eq:transfer_comp} asks for comments. To see if it may happen, let us
  take a constant source term $\fv (t) = \fv$, without turbulent fluctuation, which means
  $\fv' =\mathbf{ 0}$. In this case relation~\eqref{eq:transfer_comp} simplifies to
\begin{equation}
\label{eq:transfer_comp_2}
  \| \sqrt { \nut } D \vm \|^2 \le \| \sqrt {\E (t)} \|^2 .
\end{equation}
The usual closed equation for $k$ is
\begin{equation*}
  k_t + \vm \cdot \g k - \div (\mut \g k) = \nut | D \vm |^2 - \E,
\end{equation*} giving,
while ignoring possible boundary conditions, 
\begin{equation*}
  {d \over dt} \int_\Om k = \| \sqrt {\nut } D \vm \|^2 - \| \sqrt {\E (t)} \|^2.
\end{equation*}
Therefore,~\eqref{eq:transfer_comp_2} indicates a decrease of TKE, which means a
decrease of the turbulence, towards a laminar state, or a stable statistical
equilibrium, such as a grid turbulence.
\end{Remark}

\section{Functional setting and estimate}
\label{sec:functional-setting}
The analysis of system~\eqref{eq:Voigt} yields immediately standard a priori estimates in
$L^\infty_t L^2_x$ and $L^2_t H^1_x$, taking the solution itself as test function. The
question is whether the Voigt term $-\alpha \div (\ell D \vv_t)$ provides additional
regularity. The issue is the degeneration of the mixing length $\ell$ at the boundary,
according to~\eqref{eq:law_for_ell} below. The purpose of this section is to derive from
the interpolation theory a general estimate, that will enable us to show later that the
term $-\alpha \div (\ell D \vv_t)$ yields additional $W^{1,2}(0,T; H^{1/2})$ and
$L^\infty(0,T;H^1)$ regularity.

\subsection{Framework and preliminaries} 
As usual in mathematical fluid dynamics, we use the following spaces,
\begin{equation*}
    \begin{aligned}
      {\cal V} &= \left\{ \boldsymbol{\varphi} \in {\cal D}(\Om)^3, \, \, \g \cdot
        \boldsymbol{\varphi} = 0 {\; \; \mathrm{in}\,\, \Omega}\right\},
      \\
      H &= \left\{ \vv \in L^2(\Omega)^3, \ \nabla \cdot \vv = 0 {\; \; \mathrm{in}\,\,
          \Omega},\ \vv \cdot \mathbf{n} = 0 \text{ on }\Gamma\right\},
      \\
      V &= \left\{\vv\in H^1_{0}(\Omega)^3, \, \nabla \cdot \vv = 0 {\; \; \mathrm{in}\,\,
          \Omega}\right\},
    \end{aligned}
\end{equation*}
and we recall that ${\cal V}$ is dense in $H$ and $V$ for their respective
topologies~\cite{GR86}. Here $L^2(\Omega)$ and $H^1_0(\Omega)$ stand for the usual
Lebesgue and Sobolev spaces.  \medskip

Throughout the rest of the paper, the mixing length $\ell = \ell (\x) \in C^1 (\overline
\Om)$ is such that
\begin{equation} \label{eq:law_for_ell} 
\left\{
    \begin{array}{l} 
      \forall \, K  \subset \Om, \quad K \hbox{ compact}, \quad \zoom  \inf_K \ell >0, 
      \\
      \ell (\x ) \simeq d(\x, \Ga) = \rho(\x), \quad \hbox{ when} \quad \x \to \Ga, \,
      \text{for }\x \in \Om.
    \end{array} 
\right. 
\end{equation}
According to~\eqref{eq:law_for_ell}, we recall that 
\begin{equation*}
  H^{1 / 2} (\Om) = [H^1(\Om) , L^2(\Om)]_{1/2},
\end{equation*}
and also 
\begin{equation*} 
H^{1 / 2}_{00} (\Om) = [\huo, L^2(\Om)]_{1/2} = \left\{ u \in
    H^{1/2}(\Om), \, \hbox{ s.t. } \,\, \ell^{-{1/2}} u \in L^2(\Om) \right\}, 
\end{equation*}
cf.~\cite[Ch.~1]{MR0350177}.
In the following we will consider the following Hilbert space
 {
\begin{equation*} 
 V_{1/2} = \left\{ \vv \in H^{1/2} (\Om)^3\, ; \, \, \div \vv =
  0  \; \mathrm{in}\,\,  \Omega \quad\mathrm{and}\quad \vv \cdot \mathbf{n} =
        0 \text{ on }\Gamma \right\} 
\end{equation*}
equipped with the norm of $H^{1/2} (\Om)^3$.
}

Finally, we know that when $\Om$ is connected, the operator $D = \frac{ \g + \g^t}{2}$ is
well defined over $H^s(\Om)^3$ whatever $s \ge 0$; next (see~\cite{MR600655}),
\begin{equation*}
  K := \hbox{Ker}\,  D = \left\{ \vv \in H^s(\Om)^3 \, \hbox{ s.t.} \,\, \exists \, ({\bf a},
    {\bf b}) \in \R^3 \times \R^3; \, \, \vv (\x) = {\bf b} \times \x + {\bf a}  \right\}, 
\end{equation*}
and we recall the following Korn inequality
\begin{equation} 
  \label{eq:ine_norms} 
  \forall \, \vv \in H^1(\Om)^3, \quad \| \vv \|_{ H^1 (\Om)^3/K }\le C \| D\vv \|_{L^2 (\Om)^9},
\end{equation}
where for any given Banach space $B$ and any closed subspace $E\subset B$, $B/E$ denotes
the quotient space. Moreover, for any $ \vv \in H^1_0(\Om)^3,$ we have $ \| \vv \|_{ H^1
  (\Om)^3 }\le C \| D\vv \|_{L^2 (\Om)^9}, $ because in this case the kernel $K$ is
reduced to ${\bf 0}.$
\subsection{Main general estimate}
We deduce now the most relevant inequality, which derives from the generalized Voigt model
when using the solution itself as test function.
\begin{theorem} 
  \label{thm:main_estimate}
  Let $\vv \in \mathcal{D}'(\Omega)^3$ such that $\sqrt \ell D \vv \in
    L^2(\Om)^9$. Then $\vv \in H^{1/2}(\Om)^3$ and there exists a constant $C = C(\Om)$
    such that
\begin{equation} 
  \label{eq:norm_V1/2}  
  \|\vv \|_{H^{1/2}(\Om)^3/K} \le C \| \sqrt \ell D \vv\|_{L^2(\Om)^9}. 
\end{equation}
In particular, 
\begin{equation}\label{emb} 
  W = \left\{ \vv \in H; \, \sqrt \ell D\vv \in L^2(\Om)^9 \right\} \hookrightarrow V_{1/2},
\end{equation}
with continuous embedding. 
\end{theorem}
\begin{proof} 
We argue in two steps. 
\medskip

{\sl Step 1.} Let $\vv \in  \mathcal{D}'(\Omega)^3$ such that $\sqrt \ell D \vv \in
L^2(\Om)^9$ and $\boldsymbol \varphi \in \mathcal{D}(\Omega)^9$. As $\vv \in
H^1_{\mathrm{loc}}(\Omega)^3$, then we have 
\begin{equation*}
  |\langle D\vv, \boldsymbol{\varphi} \rangle| = \left| \int_\Omega \sqrt \ell \, D\vv
    : \dfrac{\boldsymbol{\varphi}}{\sqrt \ell} \right| \leq  C \| \sqrt \ell \, D\vv
  \|_{_{L^2(\Om)^9}} \| \boldsymbol{\varphi} \|_{H^{1/2}_{00}(\Omega)^9}.
  \end{equation*}
  Because of the density of $\mathcal{D}(\Omega)$ in $H^{1/2}_{00}(\Omega)$, this
  shows that $D\vv \in \left[H^{1/2}_{00}(\Omega)^9 \right]'$ with the estimate
\begin{equation} 
  \label{corollary-3}
  \|D \vv \|_{\left[H^{1/2}_{00}(\Omega)^9 \right]'} \leq C \| \sqrt \ell \, D \vv \|_{L^{2}(\Omega)^9}.
\end{equation}
\medskip

{\sl Step 2.} According to ~\cite{MR2247454, MR3742931}, we have
\begin{equation*} 
\forall \vv \in  L^2(\Om)^3,\quad  \| \vv \|_{L^2(\Om)^3/K} \leq C \| D\vv
  \|_{{H}^{-1}(\Omega)^9}.
\end{equation*}
Therefore, we deduce from classical interpolation theorems and from the following
identities  (see in~\cite{MR0350177}), 
\begin{equation*} 
  [H^{1}(\Omega)^3/K, L^2(\Om)^3/K ]_{1/2} =  H^{1/2}(\Om)^3/K, \quad\text{and}\quad [L^2(\Om), H^{-1}
  (\Om) ]_{1/2} =  [H^{1/2}_{00}(\Om)]',  
\end{equation*}
the inequality
\begin{equation*}
  \| \vv \|_{H^{1 / 2} (\Om)^3 /K } \le C  \|D \vv \|_{\left[H^{1/2}_{00}(\Omega)^9 \right]'}.
\end{equation*}
Hence the estimate \eqref{eq:norm_V1/2} follows by using \eqref{corollary-3} and obviously
the embedding \eqref{emb}.  
\end{proof}  
\section{Well-posedness for the generalized Navier-Stokes-Voigt equations}
\label{sec:generalized-Voigt}
In this section we start with the analysis of system~\eqref{eq:Voigt} without any eddy
viscosity, that means $\nut = 0$, both for simplicity of presentation and to highlight the
role of the generalized Voigt term. The resulting system, called generalized
Navier-Stokes-Voigt equations, is the following:
\begin{equation} \label{eq:Gen-VoigtNSE} 
\left \{ 
    \begin{aligned} 
        \vv_t  - \alpha \div (\ell \,  D \vv_t) + (\vv \cdot \nabla)\, \vv - \nu \Delta \vv
        + \nabla p  &= {\bf f} \quad\text{ in }(0, T) \times \Om, 
        \\ 
        \div \vv &= 0 \quad\text{ in }(0, T) \times \Om, 
        \\ 
        \vv \vert_\Ga &= 0 \quad\text{ on }(0, T) \times \Gamma, 
        \\ 
        \vv_{t=0} &= \vv_0 \quad\text{ in }\Om, 
\end{aligned} \right. 
\end{equation}
which is set in $Q_T = (0, T) \times \Om$, where $\Om$ is a given Lipschitz bounded domain
in $\R^3$ with its boundary $\Gamma = \partial \Om$, $T$ a fixed positive
time\footnote{Remind that when $\div \vv = 0$, then $\div (\vv \otimes \vv) = (\vv \cdot
  \g) \vv$. We use either of these forms without necessarily warning, depending on the
  situation.}, and $\ell$ satisfies~\eqref{eq:law_for_ell_1}. The main results of this
section are the existence and uniqueness of regular-weak solutions (see
Definition~\ref{def:reg_w_sol} below), when the initial velocity $\vv_0 \in V$.  \medskip

Throughout the rest of the paper, the $L^2$-norm of a given $u$ is simply denoted by
$\|u\|$, $\|\cdot\|_p $ and $\| \cdot \|_{s, p}$ denote the standard $L^p$
  and $W^{s,p}$ norms, respectively.
%
\subsection{Strong solutions} 
This aim of this subsection is to prove that given a finite time $T$, any strong
(classical) solution $\vv$ of~\eqref{eq:Gen-VoigtNSE} has natural bounds in $L^\infty(0,T;
V)\, \cap \, W^{1, 2} (0,T; V_{1/2})$ derived from energy balances, showing that the term
$-\alpha \div (\ell(\x) \, D \vv_t)$ --despite the degeneracy at the boundary-- brings a
strong regularizing effect on the system. In particular, the generalized Voigt term
provides stronger a priori estimate when compared to the usual (non regularized)
Navier-Stokes equations, since it allows to show bounds in critical scaling-invariant
spaces \textit{\`a la} Kato-Fujita. These estimates are essential for proving the
existence result of the next subsection.  \medskip

Following~\cite{Ler1934, RL19}, when considering $\vv_0 \in V \cap C (\overline \Om)^3$,
we say that $(\vv, p)$ is a strong solution to~\eqref{eq:Gen-VoigtNSE} over $Q_T= [0, T]
\times \Om$, if
\begin{itemize}
\item $\forall \, \tau <T$, $\vv \in C^2(Q_\tau)^3$, $p \in C^1(Q_\tau)$, and they satisfy
  the relations ((\ref{eq:Gen-VoigtNSE}), i), ii)) in the classical sense in $Q_\tau = [0,
  \tau] \times \Om$,
\item $\vv (t, \cdot) \in C(\overline \Om)^3$ for all $t <T$, and $\vv (t, \cdot) =0$ on
  $\Ga$,  
\item $\vv(t, \cdot)$ uniformly converges  to $\vv_0$ as $t \to 0^{+}$. 
\end{itemize} 
\begin{Remark} 
  We frequently talk about the velocity $\vv$ as a strong solution, without mentioning the
  pressure $p$. This means that we have implicitly projected the system over
  divergence-free vector fields by the Leray projector, which eliminates the pressure. The
  pressure can be  recovered via the De Rham procedure (see e.g. Temam~\cite{RT01}).
\end{Remark} 

\begin{Remark} We say that a strong solution $\vv$ of~\eqref{eq:Gen-VoigtNSE} has a
  singularity at a given time $0<T^\star<\infty$ if $\| \vv(t) \|_{\infty} \to \infty$ as $t \to
  T^\star$, $t<T^\star$. At this stage, we are not able to show that any strong solution
  has no singularity. We do not even know if there exist strong solutions, which is an
  open problem.
\end{Remark} 

The estimates we get are based on the following non standard version of Gronwall's Lemma,
the proof of which is carried out, \textit{e.g.}, in  Emmrich~\cite{EE99}.

\begin{lemma} \label{lem:Gronwall} Let $\lambda \in L^1([0,T])$, with $\lambda(t)\geq0$
  for almost all $t\in[0,T]$, let $g\in C([0,T])$ be a non-decreasing function, and let $f
  \in L^\infty ([0,T])$, such that $\forall \, t \in [0,T]$, it holds
\begin{equation*} 
f(t) \le g(t) + \int_0^t \lambda (s) f(s) \,ds. 
\end{equation*}
Then, we have
\begin{equation*} 
  f(t) \le g(t) \exp \left ( \int_0^t \lambda (s) \,ds \right ). 
\end{equation*}
\end{lemma}
In this subsection we assume that $\fv (t) = \fv \in C(\overline \Om)$ does not depend on
$t$, and we denote by $F$ either $\| {\bf f} \|_{-1, 2}^2$ or $\| {\bf f} \|^2$ so far non
risk of confusion occurs, and $C$ denotes any constant (normally $C_p \| \fv\| \le \|
\fv\|_{-1,2}$, $C_p$ being the Poincar\'e's constant). Among many choices for the
functional spaces of the source term (see also the discussion in the next subsection,
where different choices are considered), this one has the advantage that it yields a clear
and neat bound of the growth of the r.h.s in the estimates for statistical equilibrium.

The main result of this subsection is the following.
\begin{lemma} 
  \label{lem:estimates} 
  Let $\vv$ be a strong solution of~\eqref{eq:Gen-VoigtNSE} in $Q_T= [0, T] \times \Om$,
  for a given time $T$. Then, the following estimates hold true for all $s \in [0, T[$:
\begin{eqnarray} 
  \label{eq:estimate1} && 
  \| \vv(s) \|_{1/2, 2}^2 + \nu \int_0^s \| \g \vv(t) \|^2 \,dt \le
  C\left ({Fs \over \nu}  +  E(0)(\alpha, \ell) \right),  
  \\
  \label{eq:estimate11} 
  &&\nu \| \g \vv (s)  \| ^2 \le \left( \nu \| \g \vv_0 \| ^2  + Fs
  \right) \exp\left\{ \dfrac{C}{ \alpha \nu^2} \left({Fs \over \nu}  +  E(0)(\alpha, \ell)
    \right) \right\},  
\end{eqnarray} 
and
\begin{multline} 
  \label{eq:estimate111}
  C \alpha \int_0^s  \| \vv_t (t) \|_{1/2, 2}^2 \,dt + \alpha \int_0^s
  \| \sqrt \ell D \vv_t (t) \|^2 \,dt \le Fs + \nu \|\g \vv_0\|^2
  \\
  + \dfrac{C}{\alpha \nu^2} \left(\nu \|\g \vv_0\|^2 + FT\right) \left(\dfrac{Fs}{\nu} +
    E(0)(\alpha, \ell) \right) \exp\left\{ \dfrac{C}{ \alpha \nu^2} \left({FT \over \nu} +
      E(0)(\alpha, \ell) \right) \right\},
\end{multline}
where $2E(0)(\alpha, \ell) = \| \vv_0 \|^2 + \alpha \| \sqrt{\ell} D\vv_0 \|^2$. In
particular, $\vv$ has natural bounds in $L^\infty(0,T; V)$ $\cap$ $W^{1, 2} (0,T; V_{1/2})$
and $\sqrt \ell D \vv_t \in L^2(0,T; L^2(\Om)^9)$.
\end{lemma}
\begin{proof} 
  We take the dot product of (\eqref{eq:Gen-VoigtNSE}, i)) by $\vv$. We integrate by parts
  and we use the identity $\langle (\vv \cdot \nabla)\, \vv, \vv \rangle = 0$. These
  calculations are justified because $\vv$ is a strong solution, and this gives for any
  $s\in [0,T]$,
\begin{equation*} 
  \dfrac{1}{2} {d \over dt} ( \| \vv(t) \|^2 + \alpha \| \sqrt \ell D \vv (t) \|^2) + \nu
  \| \g \vv (t) \|^2 = \langle \fv (t) , \vv(t) \rangle \le {F \over 2 \nu } + {\nu \over
    2}\| \g \vv (t) \|^2 ,  
\end{equation*}
hence~\eqref{eq:estimate1} follows after having integrated in time over $[0, s]$, by
using~\eqref{eq:norm_V1/2}, the fact that the norm of $V_{1/2}$ is that inherited from
$H^{1/2}(\Om)^3$, and $H^{1/2}(\Om) \hookrightarrow L^2 (\Om)$ with continuous dense
injection.

We next take the dot product of (\eqref{eq:Gen-VoigtNSE}, i)) by $\vv_t$. In this case the
non-linear term brings a contribution in this new energy budget, given by
\begin{equation*} 
  \| \vv_t(t) \|^2 + \alpha \| \sqrt \ell D \vv_t (t) \|^2 + {\nu \over 2} {d \over dt} \|
  \g \vv (t) \|^2 = \langle \fv (t) , \vv_t(t) \rangle - \langle (\vv  \cdot \nabla)\, \vv,
  \vv_t   \rangle (t), 
\end{equation*}
As we can estimate 
$$| \langle \fv (t) , \vv_t(t) \rangle| \le \dfrac{F}{2} + \dfrac{1}{2} \|
\vv_t(t)\|^2,$$ we obtain by using~\eqref{eq:norm_V1/2}, keeping half of the contribution
of the term $\alpha \| \sqrt \ell D \vv_t (t) \|^2$,
\begin{equation} 
  \label{eq:estimate44} 
  {1 \over 2} \| \vv_t(t) \|^2 + C \alpha  \| \vv_t(t) \|_{1/2, 2}^2  + {\alpha \over 2}
  \| \sqrt \ell D \vv_t (t) \|^2 + 
  {\nu \over 2} {d \over dt} \| \g \vv(t) \|^2 \le {F \over 2} + | \langle (\vv \cdot
  \nabla)\, \vv, \vv_t \rangle (t)|. 
\end{equation} 
To deal with the nonlinear term, we use standard interpolation inequalities. The key of
the process is the continuous embedding $H^{1/2}(\Om)\hookrightarrow L^3 (\Om)^3$, which
is the limit case. Therefore, we have
\begin{equation*} 
  \begin{array}{ll} 
    |  \langle (\vv  \cdot \nabla)\, \vv, \vv_t   \rangle(t) | & \le \| \vv(t)
    \|_{6} \| \g \vv(t) \| \, \| \vv_t(t) \|_{3} 
    \\
    & \le C \| \g \vv(t)\|^2   \,  \| \vv_t(t) \|_{1/2, 2}  \phantom{\int_0^T}
    \\
    & \le \zoom {1 \over 2 C \alpha} \| \g \vv(t) \|^4 + {C \alpha \over 2}  { \| \vv_t(t)
      \|_{1/2, 2}^2 },   
  \end{array} 
\end{equation*}
so that~\eqref{eq:estimate44}  becomes 
\begin{equation}
  \label{eq:estimate35}  
  \| \vv_t(t) \|^2 + C \alpha  \| \vv_t(t) \|_{1/2, 2}^2  + \alpha \| \sqrt
  \ell D \vv_t (t) \|^2  + \nu {d \over dt} \| \g \vv (t) \|^2  \le F +  {1 \over  C
    \alpha} \| \g \vv (t) \|^4 . 
\end{equation}
In particular it follows from the above estimate that
\begin{equation*}
  \nu  {d \over dt} \| \g \vv (t) \|^2  \le F   +  {1 \over C \alpha} \| \g \vv (t) \|^4,
\end{equation*}
that we integrate over $[0, s]$, $s \in [0, T]$, so that 
\begin{equation*}
  \nu \| \g \vv (s) \|^2 \le \nu \| \g \vv_0 \|^2 + Fs + {1 \over C \alpha}\int_0^s \| \g
  \vv (t) \|^4 \,dt.
\end{equation*} 
From there, Lemma~\ref{lem:Gronwall} is applied on every time interval $[0, \tau]$ for
$\tau < T$, with
\begin{equation*}
  f(t) = \nu \| \g \vv (t) \|^2 \quad\text{and}\quad \lambda(t) = \dfrac{1}{C \alpha \nu}
  \| \g \vv (t) \|^2, 
  \end{equation*}
  both are in $L^1(0, \tau) \cap L^\infty(0, \tau)$ and $g(t) = \nu \| \g \vv_0 \|^2 + Ft
  $ which is a non decreasing function, which leads to
\begin{equation*} 
  \nu \| \g \vv (s) \|^2  \le (\nu \| \g \vv_0 \|^2 + Fs) \exp\left\{\dfrac{1}{C \alpha \nu}
    \int_0^s \| \g \vv (t) \|^2 \,dt \right\},  
\end{equation*}
and yields~\eqref{eq:estimate11} by using~\eqref{eq:estimate1}. Therefore, the
inequality~\eqref{eq:estimate111} is deduced from~\eqref{eq:estimate35} combined
with~\eqref{eq:estimate1}-\eqref{eq:estimate11}.
\end{proof}
%
\subsection{Existence and uniqueness of regular-weak solutions}
We start by giving the definition of a "regular-weak solution" to the generalized
Navier-Stokes-Voigt system~\eqref{eq:Gen-VoigtNSE}. This definition is based on
Lemma~\ref{lem:estimates}. We say "weak solution" since it is given by a weak formulation,
"regular" since, because of Lemma~\ref{lem:estimates}, we will search for a solution in
$L^\infty(0,T; V) \cap W^{1,2}(0,T; V_{1/2})$. This space is considerably smaller than
that involved in ``standard'' Leray-Hopf weak solutions to the Navier-Stokes equations
(NSE) that are just in $L^\infty(0,T; H) \cap L^2(0,T; V)$. As we shall see it, regular
weak solutions are unique and satisfy the energy equality, a fact which is still not known
about weak solutions to the NSE.
\begin{definition}
  \label{def:reg_w_sol}
  We say that a function $\vv \in L^\infty(0,T; V)
  \cap W^{1,2}(0,T; V_{1/2})$ is a regular-weak solution of the initial boundary value
  problem~\eqref{eq:Gen-VoigtNSE} if it holds true that
\begin{equation*}
    \frac{d}{dt}\Big[(\vv, \bphi) + \alpha (\ell D\vv, D\bphi)\Big] + \nu(\nabla\vv,
    \nabla\bphi) + ((\vv \cdot \nabla)\,\,\vv, \bphi)   
    = \langle\fv, \bphi \rangle \quad \forall \bphi \in V,
\end{equation*}
in the sense of $\mathcal{D}'(0,T)$ and the initial datum is attained at least in the
sense of $V_{1/2}$, that is
\begin{equation*}
 \lim_{t \to 0^+} \| \vv(t) - \vv_0 \|_{V_{1/2}} = 0. 
\end{equation*}
\end{definition}
The main theorem we prove is the following one, showing the well-posedness of the system,
globally in time. To fix the ideas and for the simplicity, we stay in a usual weak
solutions framework by taking the source term $\fv = \fv(t)$ in the space
$L^2(0,T;H^{-1/2}(\Omega)^3)$\footnote{ Recall that
  $H^{-1/2}(\Omega)=[H^{1/2}_0(\Omega)]'$ and be aware that $H^{-1/2}(\Omega) \subsetneq [
  H^{1/2}_{00}(\Om)]'$ with strict inclusion, see Lions-Magenes~\cite{MR0350177}.}.
However, many variants can be considered, starting with $\fv \in L^2(0,T;V'_{1/2})$, or
$\fv(t) = \fv \in L^2(\Om)^3$ following the previous subsection, which does not change too
much. An interesting case would be $\fv \in L^2_{uloc} (\R^{+};V'_{1/2})$, for which
additional work remains to be done in the context of the long-time behavior
(see~\cite{LBRL19}).
\begin{theorem}
  \label{thm:existence} 
  Let be given $\vv_0 \in V$ and $\fv \in L^2(0,T;H^{-1/2}(\Omega)^3)$. Then, there exists
  a unique regular-weak solution of the initial boundary value
  problem~\eqref{eq:Gen-VoigtNSE} in $[0,T]$, which satisfies the energy (of the model)
  equality for all $t \ge 0$,
\begin{equation} 
  \label{eq:energy_balance} 
  E(t)(\alpha, \ell) + \nu \int_0^t \| \nabla \vv(s) \|^2 \,ds = E(0)(\alpha, \ell) +
  \int_0^t\langle{\bf  f}(s),\vv(s)\rangle\, ds.   
\end{equation}
where $E(t)(\alpha, \ell) := {1 \over 2}  \left(\| \vv(t) \|^2  +  \alpha \| \sqrt{\ell}
  D \vv(t) \|^2\right)$. 
\end{theorem}
\begin{proof}
  The proof follows by a standard Faedo-Galerkin approximation with suitable a-priori
  estimates, compactness argument, and interpolation results. It is divided into the
  following four steps:
\begin{enumerate}[1)]
\item Construction of approximate solutions, locally in time;
\item Estimates;
\item Passing to the limit in the equations;
\item  Energy balance and uniqueness. 
\end{enumerate}
  
{\sl Step 1. Construction of approximate solutions, locally in time}. Let
$\{\bpsi_n\}_n\subset {\mathcal{V}}$ be a Hilbert basis of $V$ which we can suppose,
without lack of generality, to be orthonormal in $H$ as well as orthogonal in $V$. We look for
approximate Galerkin functions
\begin{equation*}
    \vv^n(t, x) = \sum_{j=1}^n c_{n j}(t) \,\bpsi_j(x) \qquad \text{for } n \in \N,
\end{equation*}
which has to solve the generalized Navier-Stokes-Voigt equations projected over
$\mathbf{W}_n = \text{Span}(\bpsi_1,\dots,\bpsi_n)$, that is
\begin{equation*}
    \begin{aligned}
        \frac{d}{dt}\big[(\vv^n,\bpsi_m) + \alpha(\ell
        D\vv^n,D\bpsi_m)\big] + \nu(\nabla\vv^n, \nabla\bpsi_m) + ((\vv^n\cdot\nabla)\,\vv^n,\bpsi_m)
        &=\langle\fv,\bpsi_m\rangle,
        \\
        (\vv^n(0),\bpsi_m)&=(\vv_0,\bpsi_m),
    \end{aligned}
\end{equation*}
for $m = 1, \dots, n$. The above problem is a Cauchy problem for a system of $n$-ordinary
differential equations in the coefficients $c_{n m}(t)$. We define the following
quantities for $j, l, m = 1, \dots , n$:
\begin{equation*}
    \begin{aligned}
    & \alpha_{j m} := \alpha(\ell D\bpsi_j, D\bpsi_m), \qquad
    \beta_{j m} := \nu(\nabla\bpsi_j, \nabla\bpsi_m) ,
    \\
    & \gamma_{jlm} := ((\bpsi_j\cdot\nabla) \,\bpsi_l, \bpsi_m), \qquad
    f_m(t) := \langle\fv(t), \bpsi_m\rangle,
    \end{aligned}
\end{equation*}
and we have a non-homogeneous system of ordinary differential equations with constant
coefficients (which we write with the convention of summation over repeated indices)
\begin{equation*}
    c'_{n j}(t)(\delta_{j m} + \alpha_{j m}) + c_{n j}(t) \beta_{j m} + c_{n j}(t) c_{n
      l}(t) \gamma_{jlm} = f_m(t),\qquad m=1,\dots,n,
\end{equation*}
where $\delta_{i j}$ denotes the standard Kronecker delta notation. The above system is not
in normal form. In order to obtain a system for which we can apply the Cauchy-Lipschitz
Theorem, we have to show that the matrix $(\delta_{j m}+\alpha_{j m})$ can be
inverted. Hence, since we work in a finite dimensional spaces it is enough to show that
its kernel contains only the zero vector. So let $\bxi = (\xi_{1}, \dots, \xi_{n})\in
\R^n$ be such that
\begin{equation*}
    (\delta_{j m} + \alpha_{j m}) \,\xi_j = 0.
\end{equation*}
Multiplying the above equation by $\xi_m$ and summing also over $m = 1, \dots, n$ leads to 
\begin{equation*} 
    0 =  \|\bxi\|^2 + \alpha(\ell \phi,\phi) = \|\bxi\|^2 + \alpha(\sqrt{\ell}
    \phi,\sqrt{\ell}\phi) \geq \|\bxi\|^2 \quad\text{ with } \quad  \phi := \sum_{j=1}^n
    \xi_j D\bpsi_j,  
\end{equation*}
due to the facts that $\alpha>0$ and $\ell(\x)\geq0$. Hence, this implies that
$\bxi\equiv\mathbf{0}$, hence that the matrix $(\delta_{j m} + \alpha_{j m})$ can be
inverted. This allows to rewrite the system of ODEs for the coefficients $c_{n j}$ as
follows
\begin{equation*}
  c'_{n j}(t) + c_{n j}(t) (\delta_{j m} + \alpha_{j m})^{-1} \beta_{j m} + 
  c_{n j}(t) c_{n l}(t) (\delta_{j m} + \alpha_{j m})^{-1} \gamma_{jlm} = (\delta_{j m} +
  \alpha_{j m})^{-1} f_m(t), 
\end{equation*}
and to apply the basic theory of ordinary differential equations. Note that the
coefficient from the right-hand side $f_{m}(t)=\langle\fv(t),\bpsi_m\rangle$ is not
continuous but just $L^{2}(0,T)$, hence one has to resort to an extension of the
Cauchy-Lipschitz theorem, with absolutely continuous functions, under Carath\'eodory
hypotheses (see Walter~\cite{Wal98}).
\medskip

Since the system for the coefficients $c_{n j}(t)$ is nonlinear (quadratic) we obtain that
there exists a unique solution $c_{n j}(t)\in W^{1,2}(0,T_n)$, for some $0 < T_n \leq T$.
\medskip

{\sl Step 2. Estimates}.  By taking $\vv^n$ as test function, one gets the identity
\begin{equation} 
  \label{RRB2-2.0}
    \frac{1}{2}  \dfrac{d}{dt} \left(\| \vv^n(t) \|^2  +  \alpha \| \sqrt{\ell}  D
      \vv^n(t) \|^2\right)  + \nu \| \nabla \vv^n(t) \|^2 = \langle {\bf
      f},\vv^{n}\rangle,  
\end{equation}
from which it follows
\begin{equation*} 
  \dfrac{d}{dt} \left(\| \vv^n(t) \|^2  +  \alpha \| \sqrt{\ell}  D \vv^n(t) \|^2\right)
  + \nu \| \nabla \vv^n(t) \|^2  \leq \dfrac{C_P}{\nu}\| {\bf f}\|^2_{-1/2,2},  
\end{equation*}
where $C_P = C_{P}(\Omega)$ is the Poincar\'e-type constant such that
\begin{equation*}
    \| u \|^{2}_{1/2,2} \leq C_{P} \|\nabla u \|^{2} \qquad \forall\, u \in H^{1}_{0}(\Omega).
\end{equation*}
Hence, integrating over $(0, t)$ for $t < T_n$ we get
\begin{equation}
  \label{RRB2-3}
  E^n(t)(\alpha, \ell) + \nu \int_0^t \| \nabla \vv^n(s) \|^2\,ds \leq E^n(0)(\alpha,
  \ell) + \dfrac{C_P}{\nu} \int_0^t \| {\bf f}(s)\|^2_{-1/2,2}\,ds. 
\end{equation}
where $E^n(t)(\alpha, \ell) := \| \vv^n(t) \|^2 + \alpha \| \sqrt{\ell} D
\vv^n(t)\|^2$. Next, we observe that since $\vv^n(0) \to \vv_{0}$ in $V$ and $0\leq \ell
\in C(\overline{\Omega})$, then it holds
\begin{equation*}
  \alpha \|  \sqrt{\ell} D \vv^n(0)\|^2 \leq \alpha \max_{x \in \overline{\Omega}}
  \ell(\x)\, \|\nabla\vv^n(0)\|^2 \leq  
  \alpha \max_{\x \in \overline{\Omega}} \ell(\x)\, \|\nabla\vv_{0}\|^2,
\end{equation*}
which shows that, under the given assumptions on $\vv_0$ and $\fv$ the r.h.s
of~(\ref{RRB2-3}) can be bounded independently of $n \in \N$ and consequently, a standard
continuation argument proves in fact that $T_n=T$. Moreover, it also holds
\begin{equation} 
  \label{eq:first-estimate}
  \vv^n \in L^\infty(0,T; H) \cap L^2(0,T; V) \quad \text{and} \quad \sqrt{\ell} D\vv^n
  \in L^\infty(0,T; L^2(\Omega)^9), 
\end{equation}
with norms bounded uniformly in $n\in\N$. Therefore, according to
Theorem~\ref{thm:main_estimate}, we also obtain
\begin{multline*} 
\| \vv^n(t) \|^2 + \| \vv^n(t) \|^2_{ V_{1/2} } + \int^t_0 \| \nabla \vv^n(s) \|^2 \,ds
\leq C \left [\int_0^t \| \fv (s) \|^2_{-1/2,2} ds + \| \vv_0 \|^2 + \|\nabla \vv_0 \|^2
\right ],
\end{multline*}
for a constant $C$ depending on $\nu, \alpha, \ell$ and $\Omega$. In addition, this
inequality proves that
\begin{equation*} 
  \vv^n \in L^\infty(0,T; V_{1/2}), 
\end{equation*}
with bounds independent of $n\in\N$.
\medskip

In order to give a proper meaning to the time derivative, we now use as test function
$\vv^n_t$, which is allowed, since it vanishes at the boundary and it is divergence-free.
We get
\begin{equation} 
  \label{RRB2-6}
  \| \vv^n_t(t) \|^2  +  \alpha \| \sqrt{\ell}  D \vv^n_t(t) \|^2 + \dfrac{\nu}{2}
  \dfrac{d}{dt} \| \nabla \vv^n(t) \|^2  = ({\bf f} , \vv^n_t) - ((\vv^n \cdot \nabla)\, \vv^n,
  \vv^n_t).  
\end{equation}
We estimate the r.h.s of~(\ref{RRB2-6}), thanks to the Cauchy-Schwarz,
H\"{o}lder, Young and Sobolev inequalities, which give us
$$
  |({\bf f}, \vv^n_t) |\leq C_{\epsilon} \| {\bf f}\|^2_{-1/2,2} + \epsilon \|
  \vv^n_t\|^2_{V_{1/2}}, 
$$
and
$$
  \vert ((\vv^n \cdot \nabla)\, \vv^n, \vv^n_t)\vert \leq \| \vv^n \|_6 \| \nabla\vv^n\|
  \| \vv^n_t \|_3
\leq C \| \nabla\vv^n \|^2  \| \vv^n_t \|_{V_{1/2}}
\leq C_\epsilon \| \nabla\vv^n \|^4  + \epsilon \| \vv^n_t \|^2_{V_{1/2} }. 
$$
By the above inequalities we can absorb terms in the l.h.s, to obtain
\begin{equation*} 
  C\| \vv^n_t(t) \|^2_{V_{1/2}} + \frac{\nu}{2}
  \dfrac{d}{dt} \|\nabla \vv^n(t)\|^2 \leq C_{\epsilon}\left[ \| {\bf f}(t)\|^2_{-1/2,2} +
    \| \nabla\vv^n(t) \|^4 \right],
\end{equation*}
for some $C_{\epsilon} = C(\ell, \alpha,\Omega)$. Integrating over $[0, s]$ for $s \in [0,T]$, one
obtains 
\begin{multline} 
  \label{RRB2-11} 
  C\int^s_0 \| \vv^n_t(t) \|^2_{V_{1/2}} \,dt + \dfrac{\nu}{2} \| \nabla \vv^n(s)\|^2 \leq
  \dfrac{\nu}{2} \| \nabla \vv^n(0) \|^2 + C_{\epsilon} \int^s_0 \| {\bf
    f}(t)\|^2_{-1/2,2} \,dt 
\\
  + C_\epsilon \int^s_0 \|\nabla \vv^n(t) \|^4 \,dt,
\end{multline}
hence
\begin{equation*}
  C\int^s_0 \| \vv^n_t(t) \|^2_{V_{1/2}} \,dt   + \dfrac{\nu}{2} \| \nabla \vv^n(s)\|^2
  \leq \dfrac{\nu}{2} \| \nabla \vv_0 \|^2 + C_{\epsilon} \left(\int^s_0 \| {\bf
      f}(t)\|^2_{-1/2,2} \,dt + \int^s_0 \| \nabla\vv^n(t) \|^4 \,dt \right), 
\end{equation*}
and in particular,
\begin{equation*}
    \dfrac{\nu}{2} \| \nabla \vv^n(s)\|^2  \leq \dfrac{\nu}{2} \| \nabla \vv_0 \|^2 +
    C_{\epsilon} \int^s_0 \| {\bf f}(t)\|^2_{-1/2,2} \,dt + C_\epsilon \int^s_0 \|
    \nabla \vv^n(t) \|^4 \,dt. 
\end{equation*}
We apply the Gronwall's lemma~\ref{lem:Gronwall} to get
\begin{equation} 
  \label{RRB2-13}
  \frac{\nu}{2} \| \nabla \vv^n(s) \|^2  \leq \left(\frac{\nu}{2}\|\nabla\vv_0\|^2 +
    C_{\epsilon}\int^s_0 \| {\bf f}(t)\|^2_{-1/2,2} \,dt \right)   \exp\left\{ C_\epsilon
    \int^s_0 \|\nabla \vv^n(t) \|^2 \,dt \right\},  
\end{equation}
and the r.h.s of~(\ref{RRB2-13}) is bounded uniformly in $n$ due the a priori
estimate~\eqref{eq:first-estimate}. This proves that 
\begin{equation*}
    \vv^n\in L^\infty(0,T; V),
\end{equation*}
from which we also deduce by using~\eqref{RRB2-11} that 
\begin{equation*}
    \vv^n_t \in L^2(0,T; V_{1/2}),  \quad \hbox{and therefore by~\eqref{eq:first-estimate}}
    \quad  \vv^n \in W^{1,2}(0,T;V_{1/2}),  
\end{equation*}
with uniform bounds in $n\in \N$. 
Beside estimates in $V_{1/2}$, it is important to stress that with the same track,
starting from~\eqref{RRB2-6} as in the proof of Lemma~\ref{lem:estimates}, we also have
\begin{equation*} 
  \sqrt \ell D \vv_t^n \in  L^2(0,T;L^2(\Om)^9),
\end{equation*}
again with uniform bound in $n\in\N$. 
 \medskip

 {\sl Step 3. Passing to the limit in the approximate equations}.  By the uniform bounds
 above and standard compactness results there exists $\vv \in W^{1,2}(0,T; V_{1/2} ) \cap
 L^\infty(0,T; V)$ and a sub-sequence (relabelled as $\vv^n$) such that
\begin{equation} 
\label{eq:convergence} 
\left \{ 
  \begin{aligned}
    \vv^n &\overset{*}{\rightharpoonup} \vv \qquad \qquad\, \text{in} \quad L^\infty(0,T; V),
    \\
    \sqrt{\ell} D\vv^n &\overset{*}{\rightharpoonup} \sqrt{\ell} D\vv \qquad\, \text{in} \quad
    L^\infty(0,T; L^2(\Omega)^9),
    \\
    \vv^n &\overset{}{\rightharpoonup} \vv \qquad \qquad \,\, \text{in} \quad L^p(0,T; V)
    \quad \text{ for all } 1 < p < \infty, 
    \\
    \vv^n_{t} &\overset{}{\rightharpoonup} \vv_{t} \qquad \qquad \text{in} \quad L^{2}(0,T; V_{1/2} ), 
    \\
    \sqrt{\ell} D\vv^n_t &\overset{}{\rightharpoonup} \sqrt{\ell} D\vv_t \qquad \text{in}
    \quad L^2(0,T; L^2(\Omega)^9),  
  \end{aligned} \right.
\end{equation}
To get strong convergence in appropriate spaces, we use the Aubin-Lions compactness lemma
(see~\cite{RT01}) with the triple
\begin{equation*}
    V \xhookrightarrow{} V_{3/4} \xhookrightarrow{} V_{1/2},
  \end{equation*}
  where $ V_{3/4} = [V, H]_{3/4}$, each embedding being dense and continuous. Moreover,
  since $\Omega$ is bounded by the Rellich-Kondrachov Theorem, these embeddings are also
  compact. Therefore, the sequence $(\vv^n)_{n \in \N}$ is (pre)compact in
  $L^{2}(0,T;V_{3/4})$ and (up to a sub-sequence)
\begin{equation*}
    \vv^n\rightarrow \vv \quad \text{in}\quad L^{2}(0,T; V_{3/4}),
\end{equation*}
which implies in particular strong convergence in $L^2(0,T;L^4(\Omega)^{3})$. By standard
results this allows to pass to the limit in the weak formulation, showing that indeed
$\vv$ is a regular-weak solution. We skip the details. It remains to check the initial
data. The weak convergence implies that for $0\leq t\leq T$
\begin{equation*}
  \| \vv(t) \|^2  +  \alpha \| \sqrt{\ell}  D \vv(t) \|^2   +\nu\int_0^t \| \nabla
  \vv(s) \|^2\,ds  
  \leq\| \vv(0) \|^2  +  \alpha \| \sqrt{\ell}  D \vv(0) \|^2 + \int_0^t \langle{\bf
    f}(s),\vv(s)\rangle\,ds. 
\end{equation*}
Observe that the above inequality is obtained from~\eqref{RRB2-2.0}, after integration in
time and passing to the limit. The inequality comes from the fact that $\nabla \vv^n
\rightharpoonup \nabla \vv$ in $L^2(0,T; L^2(\Omega)^9)$, and we have to consider the
inferior limit of the norm. In particular, we observe that since $\nabla \vv^n(0) \to
\nabla \vv_0$ in $L^2(\Omega)$, we can suppose, up to a further sub-sequence that $\nabla
\vv^n(0, \x) \to \nabla\vv_0(\x)$ a.e. $\x \in \Omega$, hence using the boundedness of $\ell$
and Lebesgue dominated convergence, we have
\begin{equation*}
    \|\sqrt{\ell} D\vv^n(0)\|^2 \to \|\sqrt{\ell} D\vv_0\|^2,
\end{equation*}
showing also the correct limit at the initial time.
\medskip

{\sl Step 4. Energy balance and uniqueness.} We start with the energy
balance~\eqref{eq:energy_balance}. To this end one has first to justify the use of $\vv$
as test function. From the results above, we deduce that $\vv \otimes \vv \in L^\infty
(0,T; L^3(\Om)^9)$ which yields in particular $(\vv \cdot \g) \vv \in L^2(0,T; V')$ and
$\langle (\vv \cdot \g) \vv, \vv \rangle = 0$ according to standard results. From there,
the relevant point is to check that for any $s \in [0, T]$:
\begin{equation} 
\label{eq:test-with-v}
  \int_0^s  (\ell D \vv_t, D\vv) \,dt = \frac{1}{2}\|\sqrt{\ell} D\vv(s)\|^2 -
  \frac{1}{2}\|\sqrt{\ell} D\vv_0\|^2, 
\end{equation}
since all other terms are well-behaved due to the available regularity of $\vv$. However,
$\sqrt \ell D \vv, \sqrt \ell D \vv_t \in L^2(0,T; L^2(\Omega)^9)$. Therefore, by
identifying $L^2(\Omega)^9$ with its dual space, we naturally have
\begin{equation*}
  (\ell D \vv_t, D\vv) =  \langle \sqrt \ell D \vv_t, \sqrt \ell D \vv \rangle = \frac{1}{2}{d
    \over dt} \| \sqrt \ell D \vv \|^2, 
  \end{equation*}
hence~\eqref{eq:test-with-v} and then~\eqref{eq:energy_balance} follows. 
\medskip

Moreover, this result allows us also to prove uniqueness of regular-weak solutions. In
fact, if $\vv_1$ and $\vv_2$ are solutions corresponding to the same initial datum and
same external force, taking the difference and testing (by the above argument this is
fully justified) with $\bvv = \vv_1 -\vv_2$ one obtains the following differential equality
for the difference for any $t \in [0, T]$:
\begin{equation*}
    \| \bvv(t) \|^2 + \alpha \| \sqrt{\ell}  D \bvv(t) \|^2 + \nu\int_0^t \| \nabla \bvv
    \|^2\,ds = -\int_0^t\int_\Omega(\bvv\cdot \nabla)\,\, \vv_2 \cdot \bvv \,d{\bf x} ds.  
\end{equation*}
Hence, by the usual Sobolev inequalities
\begin{equation*}
    \| \bvv(t) \|^2 + \alpha \| \sqrt{\ell}  D \bvv(t) \|^2 + \nu\int_0^t \| \nabla
    \bvv\|^2\,ds \leq \frac{\nu}{2}\int_0^t\|\nabla\bvv\|^2 \,ds + \frac{C}{\nu}
    \int_0^t\|\nabla\vv_2\|^4 \|\bvv\|^2 \,ds,  
\end{equation*}
and since $\bvv(0) = \mathbf{0}$ the Gronwall's lemma shows that $\bvv \equiv \mathbf{0}$,
due to the fact that
\begin{equation*}
    \nabla \vv_2 \in L^\infty(0,T; L^2(\Omega)^9) \subset L^4(0,T; L^{2}(\Omega)^9).
  \end{equation*}
\end{proof}
\begin{Remark} 
  The pressure is not involved in Definition~\ref{def:reg_w_sol}. However, let $(\vv_0,
  \fv)$ be given as in Theorem~\ref{thm:existence} and $\vv$ the corresponding
  regular-weak solution. Then by the De Rham theorem, we easily deduce the existence of $p
  \in {\cal D}' (0, T; L^2(\Om)/\R)$ such that $(\vv, p)$ satisfies
  System~\eqref{eq:Gen-VoigtNSE} in the sense of the distributions. The regularity of the
  pressure is probably even better than that, but this point remains to be investigated.
\end{Remark}
\begin{Remark} 
  Definition~\ref{def:reg_w_sol} is equivalent to the following: The field $\vv$ is a
  regular-weak solution to~\eqref{eq:Gen-VoigtNSE} if:
\begin{enumerate}
\item $\vv \in W^{1, 2} (0, T; V_{1/2}) \cap L^\infty (0,T; V)$, $\sqrt \ell D \vv_t \in L^2(Q_T)^9$,
\item for all $\wv \in L^2(0, T; V)$, $\forall \, s < T$:
\begin{multline*}
    \int_0^s (\vv_t, \wv) \,dt + \alpha \int_0^s  (\sqrt \ell D \vv_t, \sqrt \ell D\wv)
    \,dt - \int_0^s \int_\Om \vv \otimes \vv : \g \wv \, d\x dt 
    \\+ 
    \nu \int_0^s \int_\Om \g \vv : \g \wv  \, d\x dt =  \int_0^s \langle \fv, \wv\rangle \,dt,
\end{multline*}
\item $\lim_{t \to 0^+} \| \vv(t) - \vv_0 \|_{V_{1/2}} = 0$. 
\end{enumerate}
\end{Remark}
Once the above results of existence and uniqueness have been proved for the generalized
Navier-Stokes-Voigt equations, it is straightforward to prove the same also for the model with an
additional turbulent viscosity $\nut$ which is non-negative and bounded. We do not
reproduce here the proof, but we just present the summary as follows:
\begin{Remark} 
\label{rem:generalisation} Let $\nut \in L^\infty ([0, \infty[ \times \Om)$ such that
$\nut \ge 0$ {\sl a.e.} in $[0, \infty[ \times \Om$. We consider the initial problem
resulting from Section~\ref{sec:modelling}, with an eddy viscosity term: 
\begin{equation} \label{eq:Gen-VoigtNSE_eddy} \left\{
\begin{aligned}
    \vv_t  - \alpha \div (\ell \,  D \vv_t) + (\vv \cdot \nabla)\, \vv - \nu \Delta \vv -
    \div(\nut D \vv) + \nabla p  &= {\bf f}\quad\text{in }(0, T) \times \Om,  
    \\ 
    \div \vv &= 0 \quad\text{in }(0, T) \times \Om, 
    \\ 
    \vv \vert_\Ga &= 0 \quad\text{on }(0, T) \times \Gamma, 
    \\ 
    \vv_{t=0} &= \vv_0 \quad\text{in }\Om. 
\end{aligned}
\right. 
\end{equation}
We express  the additional eddy viscosity term $-\div (\nut D \vv)$ by 
\begin{equation*}
    - \langle \div (\nut D \vv) , \wv \rangle = (\nut D \vv, D \wv).
  \end{equation*}
  Regarding the conditions about $\nut$, the generalization of Theorem~\ref{thm:existence}
  to Problem~\eqref{eq:Gen-VoigtNSE_eddy} is straightforward, and $\vv_0\in V$ and $\fv
  \in L^2(0,T;H^{-1/2}(\Omega)^3)$ being given,~\eqref{eq:Gen-VoigtNSE_eddy} has a unique 
  regular-weak solution that satisfies the energy balance\footnote{Remind that since $\div
    \vv = 0$, then $\Delta \vv = 2 \div D\vv$. Therefore, $\langle - \nu \Delta \vv + \div
    (\nut D \vv), \wv \rangle = ( (2 \nu + \nut) D\vv, D \wv)$.}
\begin{equation*} 
  E(t)(\alpha, \ell) + \int_0^t  \| (2 \nu + \nut)^{1 /2} D \vv(s) \|^2\,ds = E(0)(\alpha,
  \ell)  + \int_0^t\langle{\bf  f}(s),\vv(s)\rangle \,ds, 
\end{equation*}
where again $E(t)(\alpha, \ell) = \dfrac{1}{2}  \left(\| \vv(t) \|^2  +  \alpha \|
  \sqrt{\ell}  D \vv(t) \|^2 \right)$.
\end{Remark}  
\section{Turbulent Voigt model involving the TKE}
\label{sec:TKE}
In this section we consider the generalized Voigt model with turbulent viscosity, coupled
with the equation for the turbulent kinetic energy, and in particular we prove a
compactness result which allows to prove existence of weak solutions.
\subsection{A compactness Lemma} 
We consider a family of models as in~\eqref{eq:Gen-VoigtNSE_eddy}, associated with
different realizations of the turbulent viscosity and study the behavior of the solutions,
under mild conditions on the given additional viscosities.

To this end let be given $(\nut^n)_{n \in \N}$ such that
\begin{equation*}
  \forall \, n \ge 0, \qquad  \nut^n \in L^\infty ([0, \infty[ \times \Om), \quad \nut^n
  \ge 0 \, \hbox{ {\sl a.e.} in } [0, \infty[ \times \Om. 
\end{equation*}
Let $\vv_0 \in V$ and $\fv \in L^2(0,T; H^{-1/2}(\Omega)^3)$. Let $(\vv^n, p^n)$ finally denote
the distributional solution to
\begin{equation} 
  \label{eq:Gen-VoigtNSE_k} 
  \left\{ \begin{aligned}
      \vv^n_t - \alpha \div
      (\ell \, D \vv^n_t) +(\vv^n \cdot \nabla)\, \vv^n - \nu \Delta \vv^n - \div (\nut^n
      D \vv^n) + \nabla p^n &= {\bf f}\quad\text{in }(0, T) \times \Om,
      \\
      \div \vv^n &= 0 \quad\text{in }(0, T) \times \Om,
      \\
      \vv^n \vert_\Ga &= 0 \quad\text{on }(0,T)\times\Gamma, 
      \\
      \vv^n_{t=0} &= \vv_0 \quad\text{in } \Om,
\end{aligned} \right. 
\end{equation}
and such that $\vv^n$ is a regular-weak solution to~\eqref{eq:Gen-VoigtNSE_k}.
 
Concerning the behavior of the solutions $\vv^{n}$ we have the following lemma.
\begin{lemma} 
  \label{thm:compact} 
  Assume that the sequence $(\nut^n)_{n \in \N}$ is uniformly bounded in $L^\infty ([0,
  \infty[ \times \Om)$ and converges almost everywhere to $\nut$ in $Q_\infty = [0,
  \infty[ \times \Om$. 

Then, it follows that: 
\begin{enumerate}[1)]
\item The sequence $(\vv^n)_{n \in \N}$ weakly converges in $W^{1,2}(0,T; V_{1/2}) \cap
  L^p(0,T; V)$, for all $p < \infty$, to a regular-weak solution $\vv$ of the limit
  problem
\begin{equation} 
  \label{eq:Gen-VoigtNSE_limit} 
  \left \{
    \begin{aligned}
      \vv_t  - \alpha \div (\ell \,  D \vv_t) + (\vv \cdot \nabla)\, \vv - \nu \Delta \vv -
      \div (\nut D \vv) + \nabla p &= {\bf f}\quad\text{in }(0, T) \times \Om,  
      \\ 
      \div \vv &= 0 \quad\text{in }(0, T) \times \Om,  
      \\ 
      \vv \vert_\Ga &= 0 \quad\text{on }(0, T) \times \Gamma,   
      \\ 
      \vv_{t=0} &= \vv_0 \quad\text{in } \Om.
    \end{aligned}
  \right. 
\end{equation}
\item The sequence $(\nut^n | D\vv^n |^2 )_{n \in \N}$ converges in the sense of measures
  to $\nut | D \vv |^2$ in $Q_T$, which means that
  \begin{equation} 
    \label{eq:conv_energy} 
    \forall \, \varphi \, \in C (\overline Q_T),
    \quad \int_0^T\int_{\Omega} \nut^n | D\vv^n |^2 \varphi \, d\x dt \underset{n \to
      \infty}{\longrightarrow} \int_0^T\int_{\Omega} \nut| D\vv |^2 \varphi \, d\x dt.
  \end{equation} 
\end{enumerate}
\end{lemma}

\begin{proof}
  In order to simplify the notation we extract sub-sequences, without changing the
  notation. However, by the uniqueness result of Theorem~\ref{thm:existence}, we finally
  get convergence for the whole sequence because of the uniqueness of solutions to the
  limit problem.
  \medskip

  1) As $\nut \ge 0$ and $\nut \in L^\infty$, we can repeat the proof of
  Theorem~\ref{thm:existence}, which yields the existence of a unique $\vv \in
  W^{1,2}(0,T;V_{1/2}) \cap L^\infty (0,T;V)$. Moreover, Theorem~\ref{thm:existence} shows
  also that each of the approximating problem has a unique solution $\vv^{n}$ such that
  the sequence of their solutions verifies~\eqref{eq:convergence}, with compactness in
  $L^2(0, T; V_{3/4})$. Passing to the limit in the equations is straightforward, except
  in the eddy viscosity term. To this end let be given $\wv \in L^2(0,T; V)$, we can write
\begin{equation*}
  \langle -\div (\nut^n D \vv^n), \wv \rangle=
  \int_{0}^{T}\int_{\Omega} \nut^n D\vv^n :
  D\wv \, d\x ds=\int_{0}^{T}\int_{\Omega} D\vv^n :
  \nut^n D\wv \, d\x ds. 
  \end{equation*}
  As $(\nut^n)_{n \in \N}$ is bounded in $L^\infty (Q_T)$, we have on one hand the
  following bound
\begin{equation*}
  | \nut^n D\wv | \le \sup_{n\in\N} \| \nut^n \|_{L^{\infty}_{t,\x}} | D \wv | \in L^2(Q_T),
  \end{equation*}
and the other hand, according to the {\sl a.e}  convergence of $\nut^n$, it follows also 
\begin{equation*}
    \nut^n D\wv  \to \nut D\wv \quad \hbox{ {\sl a.e }  in } Q_T.
  \end{equation*}
%
Then, by the Lebesgue dominated convergence theorem, one has that
\begin{equation*}
    \nut^n D\wv  \to \nut D\wv \quad \hbox{ in } L^2(Q_T).
  \end{equation*}
The convergence of the eddy viscosity term then follows from 
\begin{equation*}
    D\vv^n {\rightharpoonup} D \vv \quad \hbox{ in } L^2(Q_T),
  \end{equation*}
leading to 
\begin{equation*}
    \int_0^T\int_{\Omega} D\vv^n : \nut^n D\wv \, d\x ds \to   \int_0^T\int_{\Omega}
 D\vv : \nut D\wv \, d\x ds=   \langle - \div (\nut D \vv), \wv \rangle,
  \end{equation*}
  as $n \to \infty$. As a consequence, $\vv$ is indeed a regular-weak solution
  to~\eqref{eq:Gen-VoigtNSE_limit} on $[0,T]$, for all positive $T$.
 \medskip

 2) We split the proof into three steps:
\begin{enumerate}[i)]
\item \label{it:point_i} Weak convergence in $L^2(Q_{T})$ of the sequence 
  $( (2 \nu +\nut^n)^{1/2} D \vv^n)_{n \in \N}$ to $(2 \nu + \nut)^{1/2} D \vv$; 
\item \label{it:point_ii} Strong convergence by the ``energy method''; 
\item \label{it:point_iii}  Proof of the convergence in measures from~\eqref{eq:conv_energy}.
\end{enumerate}
\ref{it:point_i}) We already proved that the sequence $( (2 \nu + \nut^n)^{1/2} D
\vv^n)_{n \in \N}$ is bounded in $L^2(Q_T)^9$, uniformly in $n\in\N$. Moreover, we already
know that $D\vv^n {\rightharpoonup} D \vv$ in $L^2(Q_T)$. Let
\begin{equation*}
    A_n := (2 \nu + \nut^n)^{1/2} D \vv^n \qquad \text{ and } \quad A := (2 \nu +
    \nut)^{1/2} D \vv. 
  \end{equation*}
  We aim to prove that $A_n {\rightharpoonup} A$ in $L^2(Q_T)^9$. To do so, let us fix $B
  \in L^2(Q_T)^9$. By the hypotheses on $(\nut^n)_{n \in \N}$ it follows that
\begin{equation*}
    (2 \nu + \nut^n)^{1/2} B \to (2 \nu + \nut)^{1/2} B \quad  \text{ a.e. in }  Q_T.
  \end{equation*}
Moreover, one has also
\begin{equation*}
    | (2 \nu + \nut^n)^{1/2} B | \le C \left(2 \nu + \sup_n \|  \nut^n \|_\infty \right) | B | \in L^2 (Q_T).
  \end{equation*}
Therefore, again by Lebesgue's theorem we obtain 
\begin{equation*}
    (2 \nu + \nut^n)^{1/2} B \to (2 \nu + \nut)^{1/2} B \quad \text{ in }  L^2 (Q_T),
  \end{equation*}
hence
\begin{equation*}
    \int_{0}^{T} \int_{\Omega} (2 \nu + \nut^n)^{1/2} B: D \vv^n\,d\x dt \to  \int_{0}^{T}
    \int_{\Omega} (2 \nu + \nut)^{1/2} B: D \vv\,d\x dt, 
  \end{equation*}
yielding the desired weak convergence. 

\medskip

\ref{it:point_ii}) {\sl Energy method}.  We now prove the strong $L^{2}$-convergence of
the sequence $\suite A n$ to $A$. To do so, we use the energy method (see~\cite{CL14,
  MR1418142}), based on the energy (equality) balance~\eqref{eq:energy_balance} satisfied
by both $\vv^n$ and $\vv$, with the eddy viscosity terms
\begin{equation*}
    \int \int_{Q_t}  \nut | D \vv |^2 \qquad\text{and}\qquad \int \int_{Q_t}  \nut^n |D \vv^n |^2,
  \end{equation*}
in the corresponding  equation. This  means, to consider for all $t <T$, 
\begin{equation}  
  \label{energy*}
  \left\{ 
    \begin{aligned}
      E(t)(\alpha, \ell) + \int_0^t \int_\Om | A |^2 \,d\x ds &= \int_0^t \langle \fv, \vv
      \rangle \,ds +  E(0)(\alpha, \ell),  
      \\ 
      E^n(t)(\alpha, \ell) + \int_0^t \int_\Om | A_n |^2 \,d\x ds &= \int_0^t \langle \fv ,
      \vv^n\rangle \,ds + E^n(0)(\alpha, \ell), 
    \end{aligned}
  \right.
\end{equation}
where, as usual,
\begin{align*}
  E(t)(\alpha, \ell) &= \dfrac{1}{2} \left(\| \vv(t) \|^2 + \alpha \| \sqrt \ell D \vv(t)\|^2\right),
    \\
    E^n(t)(\alpha, \ell) &= \dfrac{1}{2} \left(\| \vv^n(t) \|^2 + \alpha \| \sqrt \ell D
      \vv^n(t)\|^2\right). 
\end{align*}
A critical tool is that of integrating with the respect to the time variable each equation
in~(\ref{energy*}) over $[0,T]$ and to perform then an integration by parts. This yields the
following two equalities
\begin{equation*} 
  \left\{
    \begin{aligned}
      \int_0^T E(t)(\alpha, \ell) \,dt + \int_0^T \int_\Om (T-t)  | A |^2 \, d\x dt &=
      \int_0^T  \int_0^t \langle \fv , \vv \rangle \,ds dt + T E(0)(\alpha, \ell),  
      \\
      \int_0^T E^n(t)(\alpha, \ell)\,dt + \int_0^T \int_\Om (T-t)  | A_n |^2 \, d\x dt &=
      \int_0^T  \int_0^t \langle \fv , \vv^n \rangle \,ds dt + T E^n(0)(\alpha, \ell). 
    \end{aligned}
  \right.
\end{equation*}
Arguing with the usual compactness tools as in the proof of the previous theorems we
obtain  that
\begin{equation*}
  \int_0^T \| \vv^n (t) \|^2 dt  \to \int_0^T \| \vv (t) \|^2 \,dt,
\end{equation*}
as well as 
\begin{equation*}
\begin{aligned}
    \int_0^T  \int_0^t \langle \fv , \vv^n \rangle \,ds dt  &\to  \int_0^T  \int_0^t
    \langle \fv , \vv \rangle \,ds dt 
    \\
    T E^n(0)(\alpha, \ell) &\to T E(0)(\alpha, \ell),
\end{aligned}
\end{equation*}
as $n \to \infty$. Therefore, by using the integrated energy equalities, we also get by
comparison
\begin{equation*} 
  \int_0^T  \int_\Om \left[\alpha \ell |D \vv^n |^2 + (T-t)  | A_n |^2 \right] \, d\x dt
  \underset{n \to \infty}{\longrightarrow}  \int_0^T \int_\Om \left[\alpha \ell |D \vv |^2
    + (T-t)| A |^2 \right] \, d\x dt. 
\end{equation*}
Let now $B_n$ be defined as follows 
  \begin{equation*}
B_n := (\alpha\ell + (T-t) (2 \nu + \nut^n))^{1/2} D \vv^n.
\end{equation*}
By the weak convergence result as in the previous steps, we immediately conclude that
\begin{equation*}
  B_n  \to B = (\alpha \ell + (T-t) (2 \nu + \nut))^{1/2} D \vv\quad \hbox{in }
  L^2(Q_T)^9, 
\end{equation*} 
which yields the convergence of $A_n$ to $A$ in $L^2(Q_{T'})$ for all $T'<T$. As $T$ can
be any positive time, this concludes this step.  \medskip

\ref{it:point_iii}) {\sl Proof of~\eqref{eq:conv_energy}}. By the ``Lebesgue inverse
Theorem,'' since $A_{n}\to A$ in $L^2(Q_T)$, we can extract sub-sequence, still denoted by $\suite A n$,
which converges to $A$ almost everywhere in $Q_T$, and such that there exists $G \in
L^2(Q_T)$ which satisfies
\begin{equation} 
  \label{eq:lebesgue}  
  | A (t, \x) | \le G(t,\x)  \quad \hbox{\rm { \sl a.e.} in } \, Q_T. 
\end{equation}
Let $\varphi \in C(\overline{Q_T})$, $\varphi \ge 0$, be fixed and set
\begin{equation*}
    B_n := \sqrt \varphi \sqrt {\nut^n} D \vv^n \qquad \text{ and } \qquad  B = \sqrt
    \varphi \sqrt {\nut} D \vv. 
  \end{equation*}
By using the definition of $A_{n}$ we get
\begin{equation*}
    B_n = \sqrt \varphi { \sqrt {\nut^n}  \over (2 \nu + \nut^n)^{1/2} } A_n \qquad \text{
      and } \qquad B = \sqrt \varphi { \sqrt {\nut}  \over (2 \nu + \nut)^{1/2} } A. 
  \end{equation*}
  Obviously, it follows that $B_n \to B$ {\sl a.e.} in $Q_T$, and by~\eqref{eq:lebesgue},
\begin{equation*}
    | B_n(t,\x)| \le \dfrac{1}{2 \nu}  \| \varphi \|_\infty^{1 / 2} \sup_n \| \nut^n
    \|_\infty^{1 / 2}  \,G(t,\x) \in L^2(Q_T). 
  \end{equation*}
  Therefore, $B_n \to B$ in $L^2(Q_T)$, hence~\eqref{eq:conv_energy} follows for all
  non-negative $\varphi$. The proof for all $\varphi \in C(\overline{Q_T})$ follows by
  using the splitting $\varphi = \varphi^+ - \varphi^-$, where $\varphi^+, \varphi^- \ge
  0$.
 \end{proof} 
\subsection{Application to the NSTKE-Voigt model} 
We now apply the existence result together with the compactness lemma to study the Voigt
model coupled with the equation of the turbulent kinetic energy.  The NSTKE-Voigt model is
in fact obtained by coupling the turbulent Navier-Stokes-Voigt equation to the equation
for the TKE, following the law~\eqref{eq:eddy_k}, which gives the following system:
\begin{equation} \label{eq:Gen-VoigtNSE_NSTKE} \left\{ \begin{array}{ll} \vv_t - \alpha
      \div (\ell \, D \vv_t) +(\vv \cdot \nabla)\, \vv - \nu \Delta \vv - \div (\nut (k) D
      \vv) + \nabla p = {\bf f}, & \quad (i)
      \\
      \div \vv = 0, & \quad (ii)
      \\
      \vv\vert_\Ga = 0, & \quad (iii)
      \\
      \vv_{t=0} = \vv_0, & \quad (iv)
      \\
      k_t + \vv \cdot \g k - \div( \mut (k) \g k) = \nut(k) | D \vv |^2 - (\ell+\eta)^{-1}
      k \sqrt{|k|}, & \quad (v)
      \\
      k \vert_\Ga = 0, & \quad (vi)
      \\
      k_{t=0} = k_0. & \quad (vii)
\end{array} \right. 
\end{equation}
This system calls for two comments: 
\begin{enumerate}[1)]
\item According to Lemma~\ref{thm:compact}, we know how to deal with bounded eddy
  viscosities and not better. This is why we cannot take the law~\eqref{eq:eddy_k} that we
  replace, as often in this class of problems, by
\begin{equation} 
  \label{eq:form_of_nut} 
  \nut (k) = \ell \,  T_N (\sqrt { | k |} ), 
\end{equation}
where $T_N$ is the usual ``truncation function'' at height $N$, for a given large $N \in \N$, which
is defined by for all $x \in \R$ 
\begin{equation*}
  T_N(x) := 
  \begin{cases} 
    x &\text{if } |x| \leq N,
    \\
    N \frac{x}{|x|}  &\text{if } |x| > N.
  \end{cases}
\end{equation*}
The eddy viscosity~\eqref{eq:form_of_nut} has the structure of that considered in
Lemma~\ref{thm:compact} where $\nut (k) = \ell\rho(k)$, with $\rho(k) = T_N (\sqrt { | k |} )$. Similarly, we
assume that the viscosity coefficient for the kinetic energy satisfies
\begin{equation} 
\label{eq:form_of_mut}  
\mut (k) = C \ell \,  T_{N'} (\sqrt { | k |} ), 
\end{equation}
for some dimensionless constant $C$ and another $N' \in \N$. 

\item Usually, the dissipation term in the r.h.s of the equation for $k$ should be $\E :=
  \ell^{-1} k \sqrt{|k|}$. Unfortunately, due to the degeneration of $\ell$ at the
  boundary $\Ga$, there could be further issues when passing to the limit in this term. As
  a precaution, we have approximated it by $\E= (\ell+\eta)^{-1} k \sqrt{|k|}$ where
  $\eta>0$ is a small parameter. We did not have studied yet the behavior of the solutions
  when $\eta \to 0$.
\end{enumerate} 
\begin{theorem} 
  \label{thm:existence_NSTKE} 
  Let be given $\vv_0\in V$, $\fv \in L^2(0,T;H^{-1/2}(\Omega)^3)$ and $0 \leq k_0 \in
  L^1(\Om)$. Assume that $\nut$ and $\mut$ are given by~\eqref{eq:form_of_nut}
  and~\eqref{eq:form_of_mut}. Then there exists $(\vv, k)$ such that:
\begin{enumerate} 
\item The vector field $\vv$ is a regular-weak solution to the subsystem
  [(\ref{eq:Gen-VoigtNSE_NSTKE})-(i)-(ii)-(iii)-(iv))],
\item The scalar field $k$ verifies 
\begin{equation*} 
  k \in L^\infty (0,T; L^1(\Om)), \quad k \in \bigcap_{1 < p < 5/4}  L^p(0,T; W^{1,p}(\Om)) = K_{5/4}, 
\end{equation*}
and is solution of the subsystem [(\ref{eq:Gen-VoigtNSE_NSTKE})-(v)-(vi)-(vii))] in the
sense of the distribution in $Q_T$. Moreover, $k \ge 0$ {\sl a.e.} in $Q_T$.  
\end{enumerate} 
\end{theorem}

\begin{proof} System~\eqref{eq:Gen-VoigtNSE_NSTKE} is very close to that studied
  in~\cite[Chapter 8]{CL14}. Therefore, we only indicate the changes in the proof of
  existence, without giving full details, which can be easily filled by the reader. The
  main difference is the result of the compactness Lemma~\ref{thm:compact}, which is
  essential to the proof. The further (compared to the previously studied systems)
  regularity enforced by the generalized Voigt term is the key to prove the existence
  results for the full NSTKE model.
  \medskip

  The issue is due to the quadratic source term $\nut(k) | D \vv |^2$ in the TKE equation,
  which is \textit{ a priori} in $L^1(Q_T)$ and not better.  To overcome this, we truncate
  this term as well as the initial data at height $n\in \N$, leading to the following
  regularized system:
\begin{equation} 
  \label{eq:Gen-VoigtNSE_NSTKE_nn} 
  \left\{ 
    \begin{array}{ll} 
      \vv_t  - \alpha \div (\ell \,  D \vv_t) +(\vv \cdot \nabla)\, \vv - \nu \Delta \vv -
      \div (\nut (k) D \vv)   + \nabla p  = {\bf f}, & (i)
      \\ 
      \div \vv = 0, & (ii)
      \\ 
      \vv\vert_\Ga = 0, &   (iii)
      \\ 
      \vv_{t=0} = \vv_0, &   (iv) 
      \\
      k_t + \vv \cdot \g k - \div( \mut (k) \g k) =  T_n (\nut(k) | D \vv |^2 ) -
      (\ell+\eta)^{-1} k \sqrt{|k|}, &    (v) 
      \\
      k \vert_\Ga = 0, &    (vi) 
      \\
      k_{t=0} = T_n(k_0).  &   (vii)
      \end{array} 
  \right. 
\end{equation}
For a given $\widetilde k \in L^2(0,T; \huo) \cap L^\infty (0,T; L^2(\Om))$, let $\vv = \vv(\widetilde k)$ be
the unique regular-weak solution to the subsystem
[(\ref{eq:Gen-VoigtNSE_NSTKE_nn})-(i)-(ii)-(iii)-(iv))] with $\nut (k)$ is replaced by $\nut (\widetilde k)$, so that the problem reduces to
analyze the equation for $k$, considering 
\begin{equation} 
  \label{eq:k_NSTKE_n} 
  \left\{ 
    \begin{array}{l} 
      k_t + \vv(\widetilde k) \cdot \g k - \div( \mut (k) \g k) =  T_n (\nut(k) | D
      \vv(\widetilde k) |^2 )-(\ell+\eta)^{-1} k \sqrt{|k|},   
      \\
      k \vert_\Ga = 0,     
      \\
      k_{t=0} = T_n(k_0),    
    \end{array} 
  \right. 
\end{equation}
which is a non linear parabolic equation with coefficients and a source term smooth enough
which can be easily handled. The existence of a weak solution $k \in L^2(0,T; \huo) \cap
L^\infty (0,T; L^2(\Om))$ to Problem~\eqref{eq:k_NSTKE_n}, it is easily proved and, in
addition, it follows $k_t \in L^2(0,T; H^{-1} (\Om))$.
\medskip

Finally the full system~\eqref{eq:Gen-VoigtNSE_NSTKE_nn} can be solved by
  another application of the Leray-Schauder fixed point theorem (see~\cite{CL14}) and
  again we do not know if the resulting solution is unique. We iteratively construct the solution
  starting from $k^0\equiv0$ and $\vv^{0}$ the corresponding solution of the
  [(\ref{eq:Gen-VoigtNSE_NSTKE})-(i)-(ii)-(iii)-(iv))], with $\nut=\nut(0)$. Then we
  iteratively construct the sequence  of solutions along the  following iterative scheme
  \medskip

\begin{equation} 
  \label{eq:Gen-VoigtNSE_NSTKE_n} 
  \left \{ 
    \begin{array}{l} 
      \vv_t^n  - \alpha \div (\ell \,  D \vv^n_t) + (\vv^n \cdot \nabla)\, \vv^n - \nu
      \Delta \vv^n - \div (\nut (k^n) D \vv^n)   + \nabla p^n  = {\bf f},  
      \\ 
      \div \vv^n = 0, 
      \\ 
      \vv^n\vert_\Ga = 0, 
      \\ 
      \vv^n_{t=0} = \vv_0,  
      \\
      k^n_t + \vv^{n-1} \cdot \g k^n - \div( \mut (k^n) \g k^n) =  T_n (\nut(k^n) | D \vv^{n-1}
      |^2 ) -  (\ell+\eta)^{-1} k^n \sqrt{|k^n|},   
      \\
      k^n \vert_\Ga = 0, 
      \\
      k^n_{t=0} = T_n( k_0). 
     \end{array} \right.  
\end{equation}
We are left
to pass to the limit in the above system and 
we know from~\cite[Chapter 8]{CL14} that,  up to a sub-sequence, 
\begin{equation}
  \label{eq:convergence_k} 
\left \{ 
  \begin{aligned}
    k^n &\overset{}{\rightharpoonup} k \qquad \,\, \text{in} \quad
    L^q(0,T;W_0^{1,q})\qquad \text{ for all } 1\leq q < 5/4,
    \\
    k^n_t &\overset{}{\rightharpoonup} k_t \qquad \text{in} \quad L^q(0,T;W^{-1,q})\qquad 
    \text{ for all } 1 \leq q < 5/4,
    \\
    k^n &\overset{}{\rightarrow} k \qquad \,\,\text{in} \quad L^q(Q_T) \qquad \text{ for
      all } 1\leq q < 29/14 \quad \text{and {\sl a.e.} in } Q_T.
     \end{aligned} \right.
\end{equation}
As $x \to \ell T_N \sqrt {|x|}$ is a continuous function over $\R$, $\nut^n = \nut (k^n)
\to \nut = \nut(k)$ {\sl a.e.} in $Q_T$, and because $\ell \in C^1({\overline \Om})$, we
also have $0 \le \nut \le N \| \ell \|_{\infty}$, showing that $(\nut^n)_{n \in \N}$
verifies all the requirements of Lemma~\ref{thm:compact}, by~\eqref{eq:form_of_nut}
and~\eqref{eq:convergence_k}. Therefore, $\vv^n \to \vv = \vv(k)$, the corresponding
regular-weak solution to the subsystem [(\ref{eq:Gen-VoigtNSE_NSTKE_n}), (i), (ii), (iii),
(iv))]. Passing to the limit in the equation for $k$ follows what is done in~\cite[Chapter
8]{CL14}, except about the quadratic source term. In this case, things are much better
since, according to Lemma~\ref{thm:compact}, $T_n (\nut(k^n) | D \vv^n |^2 ) \to \nut(k) |
D \vv |^2 )$ in the sense of measures. Finally, since the presence of the truncation
function obviously does not affect~\eqref{eq:conv_energy}, this ends the proof.
\end{proof}

\end{document}